\documentclass[10pt,twocolumn,twoside,journal]{IEEEtran}%

\usepackage{amsmath,amssymb,amsfonts}%
\usepackage{theorem}%
\usepackage{color}%
\usepackage{hyperref}%
\usepackage{graphicx}
\theoremstyle{change}%
\newtheorem{definition}{Definition:}[section]%
\newtheorem{proposition}[definition]{Proposition:}%
\newtheorem{theorem}[definition]{Theorem:}%
\newtheorem{lemma}[definition]{Lemma:}%
{\theorembodyfont{\rmfamily}\newtheorem{remark}[definition]{Remark:}}%

\newenvironment{proof}
  {{\bf Proof:}}
  {\qquad \hspace*{\fill} $\Box$}%

\newcommand{\N}{\mathbb{N}}%
\newcommand{\Z}{\mathbb{Z}}%
\newcommand{\R}{\mathbb{R}}%
\newcommand{\tm}{\times}%
\newcommand{\RC}{\mathcal{R}}%
\newcommand{\SC}{\mathcal{S}}%
\newcommand{\UC}{\mathcal{U}}%
\newcommand{\VC}{\mathcal{V}}%
\newcommand{\inner}{\operatorname{int}}%
\newcommand{\cl}{\operatorname{cl}}%
\newcommand{\ep}{\varepsilon}%
\newcommand{\inv}{\operatorname{inv}}%
\newcommand{\dist}{\operatorname{dist}}%
\newcommand{\bm}{\boldmath}%
\newcommand{\um}{\unboldmath}%
\newcommand{\rmS}{\mathrm{S}}%

\begin{document}

\title{Network entropy and data rates required for networked control}%

\author{Christoph Kawan and Jean-Charles Delvenne\thanks{First author: Courant Institute of Mathematical Sciences, New York University, 251 Mercer Street, New York, N.Y.~10012-1185, USA; Phone: +1 212 9983583; e-mail: kawan@cims.nyu.edu. Second author: Department of Mathematical Engineering, Universit\'e catholique de Louvain, Louvain-la-neuve (Belgium); Phone: +32 10478053; e-mail: jean-charles.delvenne@uclouvain.be.}
\thanks{The first author has been supported by DFG fellowship KA 3893/1-1.}
\thanks{The second author has been supported by the Programme of Interuniversity Attraction Poles of the Belgian Federal Science Policy Office (IAP DYSCO), and the Action de Recherches Concertees (ARC) of the Federation Wallonie-Bruxelles.}}%
\maketitle%

\begin{abstract}
We consider the problem of making a set of states invariant for a network of controlled systems. We assume that the subsystems,  initially uncoupled, must be interconnected through controllers to be designed with a constraint on the data rate obtained by every subsystem from all the other subsystems. We introduce the notion of subsystem invariance entropy, which is a measure for the smallest data rate arriving at a fixed subsystem, above which the overall system is able to achieve the control goal. Moreover, we associate to a network of $n$ subsystems a closed convex set of $\R^n$ encompassing all possible combinations of data rates within the network that guarantee the existence of corresponding feedback strategies for making a given set invariant. The extremal points of this convex set can be regarded as Pareto-optimal data rates for the control problem, expressing a trade-off between the data rates required by different systems. We characterize these quantities for linear systems, and for synchronization of chaos.%
\end{abstract}


\section{Introduction}%

A bottleneck of information, i.e., a channel transmitting information with finite data-rate capacity, inside a feedback loop may make the pursuit of a control objective more challenging or even impossible. Characterizing the required data rate to achieve a particular control task under various circumstances has been an active topic since the pioneering work of Delchamps~\cite{delchamps}. Many contributions initially focused on the case of a single system, with a single controller, and various communication constraints. Among the contributions in that setting (see, e.g., \cite{wong1999systems,tatikonda2004control,Sahai06,delv06,nair2007feedback,martins+07}), we single out Nair et al.~\cite{NEM}, characterizing the required data rate to achieve set invariance and stabilization of discrete-time deterministic systems through finite data-rate channels. In this paper, the authors introduced the notion of \emph{topological feedback entropy}, an intrinsic quantity of the open-loop system, which measures the smallest data rate above which the corresponding control problem can be accomplished by some appropriate feedback controller. In Colonius and Kawan \cite{CKa}, another quantity named \emph{invariance entropy} was introduced in the analogous continuous-time setting, as a measure for the complexity of the control task to render a set of states invariant. Though the definitions of topological feedback entropy and invariance entropy are conceptually different, it turned out that they are equivalent, after being adapted to the same (discrete-time) setting, see Colonius et al.~\cite{CKN}. In several frameworks, a key result is that achieving a control objective (such as stabilization, or making a set invariant) for a linear system of unstable eigenvalues $\lambda_1,\ldots,\lambda_k$ (and possibly other stable eigenvalues) requires a minimum data rate of $\sum_i\log|\lambda_i|$ bits per unit of time.%

Network control theory aims at the design of distributed control strategies, where the overall system is composed of several subsystems, each actuated by a specific controller. For instance, one may impose a communication graph between subsystems and controllers, with the problem to design controllers that achieve a certain control goal or minimize a control cost while respecting these interconnection patterns. Results in this direction for linear systems can be found, e.g., in the book \cite{MSa} by Matveev and Savkin.%

Along those lines, a desirable result would be, given a limited data-rate capacity $R_{ij}$ from the output of subsystem $i$ to the input of subsystem $j$ (for all pairs $i,j$), determine whether it is possible to design suitable controllers for every subsystem and communication strategies between the output of every subsystem and every controller, that achieve a certain control objective while respecting the data-rate constraints along each communication line. As far as we know, this problem is essentially open.%

In this paper, we tackle a simpler problem, where a constraint is put on the \emph{total} data rate accessible to input of each subsystem. This limit on the data rate can be seen as a bottleneck of information at the entry of the subsystem. One can assume for instance that only an imperfect, e.g., quantized, measurement is accessible to the controller, which then decides of the input to apply to the subsystem. Equivalently, one can assume, as we do in this paper, that the bottleneck stands between the controller (seen as a coder, in a coding-theoretic view) having perfect knowledge of the overall state and the actuator (decoder). The problem is therefore to design a set of controllers achieving a certain control goal given these data-rate constraints. Note that subsystems only communicate through the controllers we design, i.e., do not bypass the bottleneck of information through direct connections (see Fig.~\ref{fig1}). We assume that the goal is to make a certain subset $Q$ of the overall state space $X_1 \tm \ldots \tm X_n$ (where $X_i$ is the state space of subsystem $i$) invariant.%

As an example, one may think of drones, or other kinds of agents that must maintain a certain shape in space, e.g., so that every distance $\|x_i-x_j\|$ is in a prescribed interval $d_{i,j} \pm \ep$. The positions are measured, e.g., with cameras by a central entity, and a centrally-computed appropriate control signal is sent to each drone through a finite rate wireless channel (which stands here between control and actuation). Alternatively, the central entity only sends quantized estimates of the overall state to each drone, which then computes the most appropriate course of action (the channel stands here between estimation and controller).%

In this paper, we characterize the set of possible data rates that must be received at the entry of each subsystem, by a suitable generalization of invariance entropy \cite{CKa}, called the network entropy set. It is a subset of $\R^n$, that depends on the $n$ individual subsystems and the set $Q$ to be made invariant. We show that a point $(h_1,\ldots,h_n)$ belongs to this set if and only if there is a control strategy that achieves the control objective, where the first subsystem receives a data rate $h_1$, the second subsystem a data rate $h_2$, etc.%

We find that in some situations there is a trade-off between the rates to be allowed to the systems: one subsystem can receive no information at all if the other receives twice more, for instance. This is the case when  chaotic systems are to be practically synchronized, i.e., interconnected so that their trajectories remain within distance $\ep$ from one another. In other cases, such as controllable linear systems, there is no such trade-off: the control goal is achievable if and only if a sufficient rate is available to each of the subsystems, whose minimum value only depends on this specific subsystem.%

A simpler case is when only one of the subsystems obeys a data-rate constraint, while the other subsystems have full access to the state of every subsystem. We characterize the minimum required data rate for this subsystem as the subsystem invariance entropy. We recover invariance entropy in case of a single system ($n=1$). We also show that the subsystem invariance entropy takes, under mild conditions, the form $\sum_i \log|\lambda_i|$, summed over unstable eigenvalues, for linear subsystems.%

It should be noted that the kind of channels we consider here can be deterministic (lossless transmission of a finite alphabet of symbols), nondeterministic (possible confusion between two symbols), but not stochastic, as this would require a different, probabilistic statement of the control goal. The data-rate capacity is therefore defined as the zero-error capacity for the channel. We assume here that the transmission through the channel can occur without transmission or decoding delay. Of course, the existence of such delays would make the bounds we find in this paper conservative, instead of tight. In the presence of delays, capacity of a channel should be replaced by anytime capacity~\cite{Sahai06}.%

In this paper, we work with discrete-time systems described by difference equations, a time interval $[0,\tau]$ being understood as the set of nonnegative integers less than or equal to $\tau$. However, the general definitions and results can easily be adapted to continuous-time systems, described by differential equations, where $[0,\tau]$ now denotes a real interval.%

The paper is organized as follows. In Section \ref{sec:1}, we recall from \cite{CKa} the concept of invariance entropy in the case of a single system. Section \ref{sec:2} defines overall system and subsystem invariance entropy, as well as main properties of the latter, including the connection with the required data rate to achieve the control objective. Subsystem invariance entropy for linear systems is derived in Section \ref{sec:lin1}. The network entropy set, its definition, properties, including relationship with subsystem invariance entropy, is the object of Section \ref{sec:3}. In Section \ref{sec:4}, the network entropy set for linear systems is characterized, while the network entropy set for synchronization of chaotic systems is treated in Section \ref{sec:5}. We end with perspectives.%
   
\begin{figure}[!h]
\centering
\includegraphics[width=8cm]{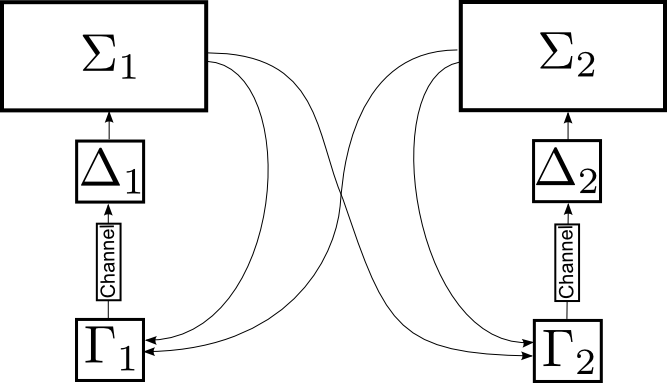}
\caption{A networked system is composed of $n$ subsystems $\Sigma_i$, here $n=2$. A limited data-rate capacity channel takes place between a coder $\Gamma_i$ and a decoder $\Delta_i$. The coder/decoder pair may be understood, e.g., as a quantizer/controller pair, or a controller/actuator pair (as we assume in this paper), etc. The problem is to determine which zero-error data-rate capacities allow the control objective (making a certain compact set $Q$ invariant in the overall state space) to be achieved, for some control and actuation strategies $\Gamma_i,\Delta_i$.}\label{fig1}
\end{figure} 

{\bf Notation:} We write $\Z$ for the set of integers and $\Z_+$ for the set of nonnegative integers. Logarithms are assumed to be taken to the base $2$.%

\section{Control systems and invariance entropy}\label{sec:1}%

In this paper, we consider discrete-time control systems given by difference equations%
\begin{equation*}
  x_{k+1} = f(x_k,u_k),\quad k\geq0.%
\end{equation*}
Here the right-hand side is a map $f:X\tm U \rightarrow X$, where $X$ is a topological space (the \emph{state space} of the system) and $U$ a nonempty set (the \emph{control value set}). We assume that for each $u\in U$ the map $f_u:X\rightarrow X$, $x\mapsto f(x,u)$, is continuous. The admissible control sequences are the elements of $\UC := U^{\Z_+}$, and the dynamics of the system is described by the \emph{transition map} $\varphi:\Z_+ \tm X \tm \UC \rightarrow X$,%
\begin{equation*}
  \varphi(k,x,\omega) = \left\{\begin{array}{cc}
                                    x & \mbox{if } k=0,\\
                                    f_{\omega_{k-1}} \circ \cdots \circ f_{\omega_0}(x) & \mbox{if } k\geq1.%
                               \end{array}\right.%
\end{equation*}
Note that for each $k\in\Z_+$ and $\omega\in\UC$ the map $\varphi_{k,\omega}:X\rightarrow X$, $x\mapsto\varphi(k,x,\omega)$, is continuous.%

We call a compact set $Q \subset X$ with nonempty interior \emph{(strongly) controlled invariant} provided that for every $x\in Q$ there exists $u\in U$ such that $f_u(x) \in \inner Q$. Given such a set $Q$, we define the \emph{invariance entropy of $Q$} as follows. For $\tau>0$, a set $\SC\subset\UC$ of control sequences is called $(\tau,Q)$-spanning if for every $x\in Q$ there is $\omega\in\SC$ with%
\begin{equation*}
  \varphi(k,x,\omega) \in \inner Q \mbox{\quad for\ } k=1,\ldots,\tau.%
\end{equation*}
We let $r_{\inv}(\tau,Q)$ denote the minimal cardinality of a $(\tau,Q)$-spanning set and define the invariance entropy of $Q$ by%
\begin{equation*}
  h_{\inv}(Q) := \lim_{\tau\rightarrow\infty}\frac{1}{\tau}\log r_{\inv}(\tau,Q).%
\end{equation*}
As shown in \cite{CKN}, the numbers $r_{\inv}(\tau,Q)$ are finite and the limit exists because of subadditivity (and hence is equal to the infimum over $\tau>0$). If we consider more than one system at the same time, we sometimes write $h_{\inv}(Q;\Sigma)$ to refer to a specific system $\Sigma$.%

In \cite{CKN} it has been shown that the quantity $h_{\inv}(Q)$ coincides with the topological feedback entropy introduced by Nair et al.~\cite{NEM}. Hence, it is a measure for the smallest data rate in a channel between coder and controller, above which the system is able to render the set $Q$ invariant, a typical goal in control theory. In a metric space setting, the definition of topological feedback entropy can be modified in such a way that it becomes an analogous measure for the problem of local uniform exponential stabilization at an equilibrium point. This is done by taking appropriate limits, letting the size of the set $Q$ and that of the control range tend to zero. In Nair et al.~\cite{NEM} it is proved that the corresponding data rate or entropy can be expressed in terms of the unstable eigenvalues of the linearization about the equilibrium. Similar formulas and estimates for the invariance entropy can be found in the monograph \cite{Kaw}.%

\section{Subsystem invariance entropy}\label{sec:2}%

As a step towards characterizing the data rate required for each of $n$ interacting subsystems cooperating to achieve a common control goal for the overall system, we study the particular case where only the actuator of the $i$-th subsystem receives a constrained data rate, while other subsystems obey no such constraint and can take advantage of full knowledge about the overall state. The minimum data rate required is shown to be appropriately modeled by the subsystem invariance entropy, introduced in this section.%

\subsection{Definition and elementary properties}%

Consider a discrete-time control system $\Sigma$ which is the direct product of $n$ subsystems $\Sigma_1,\ldots,\Sigma_n$. We write $X_i$ for the state space and $U_i$ for the set of control values of $\Sigma_i$. The dynamics is given by%
\begin{equation*}
  x^{(i)}_{k+1} = f_i\left(x_k^{(i)},u_k^{(i)}\right),\quad k\geq0.%
\end{equation*}
We assume that $X_i$ is a topological space, $U_i$ a nonempty set, and $f_i:X_i\tm U_i\rightarrow X_i$ a map which is continuous in its first component. We write $\varphi_i:\Z_+\tm X_i\tm \UC_i \rightarrow X_i$ for the associated transition map, where $\UC_i := U_i^{\Z_+}$. The state space of the \emph{overall system} $\Sigma$ is the Cartesian product $X = X_1\tm \cdots \tm X_n$ (endowed with the product topology) and the control value set is $U = U_1\tm\cdots\tm U_n$. The corresponding transition map is given by $\varphi(k,x,\omega) = (\varphi_1(k,x_1,\omega_1),\ldots,\varphi_n(k,x_n,\omega_n))$, $\varphi:\Z_+ \tm X \tm \UC \rightarrow X$, $\UC = \UC_1\tm\cdots\tm\UC_n$. Moreover, we denote by $\pi_i:X\rightarrow X_i$ the canonical projection to the $i$-th component. Note that this map is continuous and open. For the projection to the $i$-th component of the space of control sequences we write $\pi_{\UC_i}:\UC \rightarrow \UC_i$.%

A system of this type can be a model for the underlying dynamics of a multi-agent system, in which the uncoupled subsystems are supposed to satisfy a common goal. An example would be a platoon of vehicles, where the vehicles should follow a common leader with the same velocity and prescribed distances. Another example are cooperating robots that are supposed to distribute over some region to get measurements, or to meet at a common place (see, e.g., \cite{Lun}). The following definition introduces a notion of entropy related to the control aim of keeping the overall system in a prescribed subset of the state space. In the vehicle example, this subset might be chosen in such a way that the distance of two consecutive vehicles is kept within a certain interval and also the velocities stay in a certain interval.%

\begin{definition}
Given a controlled invariant set $Q$ of $\Sigma$, $i\in\{1,\ldots,n\}$, and $\tau>0$, a subset $\SC_i \subset \UC_i$ is called $(\tau,Q)^{(i)}$-spanning if the set $\UC_1 \tm \cdots \tm \UC_{i-1} \tm \SC_i\tm\UC_{i+1}\tm \cdots \tm \UC_n$ is $(\tau,Q)$-spanning. The minimal cardinality of such a set is denoted by $r_{\inv}^{(i)}(\tau,Q)$ and we define the {\bf \bm$i$\um-th subsystem invariance entropy} of $Q$ by%
\begin{equation*}
  h_{\inv}^{(i)}(Q) := \lim_{\tau\rightarrow\infty}\frac{1}{\tau}\log r_{\inv}^{(i)}(\tau,Q).%
\end{equation*}
\end{definition}

In other words, for each $\tau$, we seek among all $(\tau,Q)$-spanning sets $\SC$ one whose projection to the $i$-th component, $\pi_{\UC_i}\SC$, has smallest cardinality. The asymptotic growth rate of this cardinality as $\tau \rightarrow \infty$ is the $i$-th subsystem invariance entropy.%

The following proposition shows that $h_{\inv}^{(i)}(Q)$ is well-defined and summarizes some of its elementary properties.%

\begin{proposition}\label{prop_elprops}
Let $Q$ be a controlled invariant set of $\Sigma$ and fix $i\in\{1,\ldots,n\}$. Then the following statements hold:%
\begin{enumerate}
\item[(a)] The numbers $r_{\inv}^{(i)}(\tau,Q)$ are finite and the sequence $\tau \mapsto \log r_{\inv}^{(i)}(\tau,Q)$ is subadditive. Therefore,%
\begin{equation}\label{eq_infeq}
  h_{\inv}^{(i)}(Q) = \inf_{\tau>0}\frac{1}{\tau}\log r_{\inv}^{(i)}(\tau,Q) < \infty.%
\end{equation}
\item[(b)] If $Q$ is a Cartesian product of compact sets $Q_j \subset X_j$ with nonempty interiors, $Q = Q_1 \tm \cdots \tm Q_n$, then $Q_i$ is a controlled invariant set of $\Sigma_i$ and%
\begin{equation}\label{eq_cpif}
  h_{\inv}^{(i)}(Q) = h_{\inv}(Q_i).%
\end{equation}
\item[(c)] In general, $\pi_i(Q)\subset X_i$ is a controlled invariant set of $\Sigma_i$ and%
\begin{equation}\label{eq_estimates}
  h_{\inv}(\pi_i(Q);\Sigma_i) \leq h_{\inv}^{(i)}(Q) \leq h_{\inv}(Q;\Sigma).%
\end{equation}
\end{enumerate}
\end{proposition}

\begin{proof}
To show (a), note that finiteness of $r_{\inv}^{(i)}(\tau,Q)$ follows from the simple observation that a $(\tau,Q)$-spanning set $\SC\subset\UC$ projects to a $(\tau,Q)^{(i)}$-spanning set (cf.~the proof of (c)), and $\SC$ can be chosen to be finite, which follows from continuity of the transition map with respect to $x$ and compactness of $Q$ (see \cite[Prop.~2.2]{NEM}). To show subadditivity, take a $(\tau_1,Q)^{(i)}$-spanning set $\SC_i^1 \subset \UC_i$ and a $(\tau_2,Q)^{(i)}$-spanning set $\SC_i^2 \subset \UC_i$ for arbitrary $\tau_1,\tau_2>0$. Define a new set $\SC_i \subset \UC_i$ by%
\begin{equation*}
  \SC_i := \left\{ \omega \star \mu:\ \omega \in \SC_i^1,\ \mu\in\SC_i^2 \right\},%
\end{equation*}
where $\omega \star \mu$ is defined as the concatenation%
\begin{equation*}
  (\omega \star \mu)_k = \left\{\begin{array}{cc}
                                    \omega_k & \mbox{for } 0 \leq k \leq \tau_1-1\\
                                    \mu_{k-\tau_1} & \mbox{for } \tau_1 \leq k \leq \tau_1+\tau_2-1%
                                 \end{array}\right.,%
\end{equation*}
(extended arbitrarily for $k\geq\tau_1+\tau_2$.) We claim that $\SC_i$ is a $(\tau_1+\tau_2,Q)^{(i)}$-spanning set. Indeed, take $x\in Q$. Then there exists $\omega \in \UC_1 \tm \cdots \tm \UC_{i-1} \tm \SC_i^1 \tm \UC_{i+1}\tm\cdots\tm\UC_n$ with $\varphi(k,x,\omega) \in \inner Q$ for $k=1,\ldots,\tau_1$. Put $y := \varphi(\tau_1,x,\omega) \in Q$. Then there is $\mu \in \UC_1 \tm \cdots \tm \UC_{i-1} \tm \SC_i^2 \tm \UC_{i+1}\tm\cdots\tm\UC_n$ with $\varphi(k,y,\mu) \in \inner Q$ for $k=1,\ldots,\tau_2$, or equivalently, $\varphi(k+\tau_1,x,\omega \star \mu) \in \inner Q$ for $k=1,\ldots,\tau_2$. Since the $i$-th component of $\omega\star\mu$ is contained in $\SC_i$, this proves the claim. Choosing $\SC_i^1$, $\SC_i^2$ minimal, it follows that $r_{\inv}^{(i)}(\tau_1+\tau_2,Q) \leq r_{\inv}^{(i)}(\tau_1,Q) \cdot r_{\inv}^{(i)}(\tau_2,Q)$, implying subadditivity of $\log r_{\inv}^{(i)}(\tau,Q)$. The equality in \eqref{eq_infeq} now follows from Fekete's subadditivity lemma (cf.~\cite[Lem.~2.1]{CKN}).%

To show (b), take $x_i \in Q_i$ and let $x\in Q$ with $\pi_i(x) = x_i$. Since $Q$ is controlled invariant, there exists $u = (u_1,\ldots,u_n) \in U_1 \tm \cdots \tm U_n$ with $f(x,u) \in \inner Q$. This implies%
\begin{equation*}
  f_i(x_i,u_i) = \pi_i(f(x,u)) \in \pi_i(\inner Q) \subset \inner\pi_i(Q) = \inner Q_i,%
\end{equation*}
since $\pi_i$ is an open map. Hence, $Q_i$ is controlled invariant with respect to $\Sigma_i$. To show \eqref{eq_cpif}, assume that $\SC_i \subset \UC_i$ is $(\tau,Q_i)$-spanning. We claim that $\SC_i$ is also $(\tau,Q)^{(i)}$-spanning. Indeed, for every $x = (x_1,\ldots,x_n) \in Q = Q_1 \tm\cdots \tm Q_n$, we find $\omega_j \in \UC_j$ ($j\neq i$) with $\varphi_j(k,x_j,\omega_j) \in \inner Q_j$ for $k=1,\ldots,\tau$ by controlled invariance of $Q_j$, and $\omega_i \in \SC_i$ with $\varphi(k,x_i,\omega_i) \in \inner Q_i$ for $k=1,\ldots,\tau$. Putting $\omega := (\omega_1,\ldots,\omega_n)$, we find%
\begin{equation*}
  \varphi(k,x,\omega) \in \inner Q_1 \tm \cdots \tm \inner Q_n = \inner Q \mbox{\ for\ } k=1,\ldots,\tau,%
\end{equation*}
proving the claim. On the other hand, if $\SC_i\subset\UC_i$ is $(\tau,Q)^{(i)}$-spanning, then it is obviously also $(\tau,Q_i)$-spanning. Hence, $(\tau,Q_i)$-spanning and $(\tau,Q)^{(i)}$-spanning sets are in one-to-one correspondence, implying \eqref{eq_cpif}.%

Finally, let us show (c). Since $\pi_i$ is continuous and open, $\pi_i(Q)$ is compact and has nonempty interior. The proof of controlled invariance is the same as in (b). With the same reasoning as before, we see that a $(\tau,Q)^{(i)}$-spanning set $\SC_i\subset\UC_i$ is also $(\tau,\pi_i(Q))$-spanning. This implies the first inequality in \eqref{eq_estimates}. To see the second one, take a $(\tau,Q)$-spanning set $\SC \subset \UC = \UC_1\tm\cdots\tm\UC_n$ and put $\SC_i := \pi_{\UC_i}(\SC)$. We claim that $\SC_i$ is $(\tau,Q)^{(i)}$-spanning. Indeed, take $x\in Q$. Then there exists $\omega \in \SC$ with $\varphi(k,x,\omega) \in \inner Q$ for $k=1,\ldots,\tau$. Since $\omega = (\omega_1,\ldots,\omega_n) \in \SC \subset \UC_1 \tm \cdots \tm \UC_{i-1} \tm \pi_{\UC_i}(\SC) \tm \UC_{i+1}\tm\cdots\tm\UC_n$, this proves the claim and completes the proof of (c).%
\end{proof}

\begin{remark}
Notice that the obvious monotonicity properties of $r_{\inv}^{(i)}(\cdot)$ with respect to $\tau$ and to each $U_j$ hold, i.e., $\tau \mapsto r_{\inv}^{(i)}(\tau,Q)$ is increasing, and enlarging any of the control value sets $U_j$ can only lower the values of $r_{\inv}^{(i)}(\tau,Q)$ and hence of $h_{\inv}^{(i)}(Q)$.%
\end{remark}

\subsection{The data-rate theorem}%

In this section, we prove that the $i$-th subsystem invariance entropy $h_{\inv}^{(i)}(Q)$ measures the smallest possible information rate, more precisely the zero-error capacity, at the entry of  $i$-th subsystem above which the overall system is able to render the set $Q$ invariant,  while the other subsystems can be controlled with full knowledge of the overall state.%

Remember that we assume for convenience in this paper that the bottleneck of information stands between the controller (assumed to possess full knowledge of the overall state) and the actuator. The controller generates a signal over the time interval $(0,\tau] = \{1,\ldots,\tau\}$ described by $\Gamma_{\tau}:X^{(0,\tau]}\to B^{(0,\tau]}$, where $B$ is an alphabet used for transmission into the channel. The channel transmits the signal as a possibly nondeterministic (set-valued) map $\kappa_\tau:B^{(0,\tau]} \to B^{(0,\tau]}$. The actuator, reading the (possibly corrupted) signal from the channel, acts on the system with an input signal given by the map $\Delta_{\tau}:B^{(0,\tau]} \to U^{(0,\tau]}$.%

\begin{remark} 
The maps $\Gamma$ and $\Delta$ could in principle be chosen to be nondeterministic (set-valued), however it is easy to see that for all nondeterministic maps $\Gamma,\Delta$ achieving a control objective, deterministic maps can be chosen instead that achieve the same control objective. Thus, there is no loss of generality in assuming $\Gamma,\Delta$ to be deterministic as we do.%
\end{remark}

The zero-error capacity of such a channel is given by $\lim(1/\tau)\log b_{\tau}$, where $b_{\tau}$ is the maximum cardinality of a subset of $B^{(0,\tau]}$ whose elements are pairwise distinguishable when sent through the channel. Two signals in $B^{(0,\tau]}$ are distinguishable if their images under $\kappa_\tau$ have empty intersection. Therefore, the zero-error capacity is the maximum data rate that can be reliably transmitted through the channel.%


In this context, one can state the following data rate theorem.%

\begin{theorem}\label{thm:datarate1}
Let $Q$ be a controlled invariant set of $\Sigma$ and fix $i\in\{1,\ldots,n\}$. Then $h^{(i)}_{\inv}(Q)$ is the infimum zero-error capacity required between the $i$-th controller and actuator of $\Sigma_i$ over all overall control strategies $\Gamma,\Delta$ that make $Q$ invariant.%
\end{theorem}

\begin{proof}
Over a time interval, any successful control strategy $\Gamma,\Delta$ must be such that the image of $\Delta^{(i)}_\tau$ is at least of cardinality $r^{(i)}_{\inv}(\tau,Q)$. As the control objective, making $Q$ invariant, must succeed whatever corruption occurs in the channel, the same control strategy must be successful for any deterministic version of the channel $\tilde{\kappa}_\tau: B^{(0,\tau]} \to B^{(0,\tau]}$, i.e., a deterministic map created by choosing arbitrarily $\tilde{\kappa}_\tau(s)$ among the sets $\kappa_\tau(s)$ for every channel signal $s \in B^{(0,\tau]}$. The minimal cardinality of the image of $\tilde{\kappa}_\tau(s)$ is precisely $b_{\tau}$, as can easily be seen. For one such minimal choice of $\tilde{\kappa}_\tau(s)$, one can modify $\Gamma^{(i)}_\tau$ to $\tilde{\Gamma}^{(i)}_{\tau}$ so that $\Delta_{\tau} \circ \tilde{\kappa}_{\tau}$ is injective on the image of $\tilde{\Gamma}_{\tau}$. Indeed, if two channel signals $s_1,s_2 \in B^{(0,\tau]}$ lead to the same final input signal $\Delta_{\tau}(\tilde{\kappa}_\tau(s_1))=\Delta_{\tau}(\tilde{\kappa}_\tau(s_2))$, then the controller $\Gamma^{(i)}_\tau$ may as well be replaced by $\tilde{\Gamma}=\Gamma^{(i)}_\tau \circ \tilde{\kappa}_\tau \circ \tilde{\kappa}_\tau^{-1}$, for some choice of an inverse $\tilde{\kappa}_\tau^{-1}$, so as to comprise only $s_1$ or $s_2$ in its image, with the same final input signal being delivered to the system. In summary, we derive a modified control strategy $(\tilde{\Gamma}_\tau, \Delta_\tau)$ able to make $Q$ invariant through channel $\tilde{\kappa}_\tau$ until time $\tau$ at least, generating at most $s_\tau$ different input signals for $\Sigma_i$. Since this strategy is successful in making $Q$ invariant, the set of those input signals must be $(\tau,Q)^{(i)}$-spanning, thus must be of cardinality at least $r^{(i)}_{\inv}(\tau,Q)$, and therefore $b_\tau \geq r^{(i)}_{\inv}(\tau,Q)$. Passing to limit of large $\tau$, we see that the zero-error capacity of the channel is at least $h^{(i)}_{\inv}(Q)$. 

We need to prove that $h^{(i)}_{\inv}(Q)$ can be reached as an infimum of all allowed capacities, for some control strategies $\Gamma, \Delta$ and some channels $\kappa$. For any $\ep >0$, consider a $\tau$ large enough so that $\log r^{(i)}_{\inv}(\tau,Q)/\tau < h^{(i)}_{\inv}(Q) + \ep$. One chooses a $(\tau,Q)^{(i)}$-spanning set $\SC_i$. Then one can devise a block-coding strategy for $\Gamma,\Delta$, that measures $x_0$, then transmits through a no-delay channel the index of an appropriate element of $\SC_i$ that will maintain $Q$ invariant until time $\tau$. At time $\tau$, a measurement of $X_{\tau}$ is made by the controller, which then transmits the index of an appropriate element of $\SC_i$ to the actuator, that will maintain $Q$ invariant until $2\tau$, etc.%
\end{proof}

\subsection{Transformations}%

In this subsection, we describe a class of transformations preserving the subsystem invariance entropy. We know that invariance entropy is an invariant with respect to state transformations (see, e.g., \cite[Thm.~3.5]{CKa}), but not with respect to feedback transformations, which can be seen by looking at the formula for the entropy of linear systems that involves eigenvalues, not preserved by feedback transformations. The following proposition shows that this is different for subsystem invariance entropy. Here feedback transformations applied to all subsystems $\Sigma_j$, $j\neq i$, leave $h^{(i)}_{\inv}(Q)$ unchanged, whereas for $\Sigma_i$ only state transformations are allowed.%

In general, a (topological) state transformation of a system $x_{k+1} = f(x_k,u_k)$ with state space $X$ is given by a homeomorphism $\alpha:X\rightarrow Y$ onto a space $Y$. Then the dynamics of the transformed system on $Y$ is described by%
\begin{equation*}
  y_{k+1} =  g(y_k,u_k),\quad g(y,u) = \alpha(f(\alpha^{-1}(y),u)).%
\end{equation*}
Consequently, the $f$-trajectory with initial value $x$ and control sequence $u_k$ is transformed by $\alpha$ into the $g$-trajectory with initial value $\alpha(x)$ and the same control sequence. Additionally, we will allow a (bijective) transformation $\beta:U \rightarrow V$ of the control value set, in which case the transformed system takes the form%
\begin{equation*}
  y_{k+1} =  g(y_k,v_k),\quad g(y,v) = \alpha(f(\alpha^{-1}(y),\beta^{-1}(v))).%
\end{equation*}
We will also call these more general transformations $\alpha \tm \beta:X\tm U \rightarrow Y\tm V$ \emph{state transformations}.%

In contrast, a feedback transformation does not act on the state and control variables separately, since here the transformation of the control variable may also depend on the state. A \emph{feedback transformation} of the system $x_{k+1} = f(x_k,u_k)$ is given by a bijection $\Phi:X \tm U \rightarrow Y \tm V$ of the form $\Phi(x,u) = (\gamma(x),\delta(x,u))$, where $\gamma:X \rightarrow Y$ is a homeomorphism. In this case, the new right-hand side $g:Y\tm V \rightarrow Y$ is related to the old one by%
\begin{equation*}
  \gamma(f(x,u)) = g(\gamma(x),\delta(x,u)),%
\end{equation*}
and the $f$-trajectory $x_k$ with initial value $x_0$ and control sequence $u_k$ is mapped by $\gamma$ to the $g$-trajectory with initial value $\gamma(x_0)$ and control sequence $\delta(x_k,u_k)$.%

For simplicity we will assume that $n=2$ in the following, which we can do without loss of generality, since for a fixed $i\in\{1,\ldots,n\}$ we can combine the subsystems $\Sigma_j$, $j\neq i$, to one larger subsystem.%

\begin{proposition}\label{prop_equivalence}
Consider two networked systems given by%
\begin{equation}\label{eq_sys1}%
  \begin{array}{rl} x^{(1)}_{k+1} & = f_1\left(x_k^{(1)},u_k^{(1)}\right),\qquad (u_k^{(1)}) \in \UC_1\\
  x^{(2)}_{k+1} & = f_2\left(x_k^{(2)},u_k^{(2)}\right),\qquad (u_k^{(2)}) \in \UC_2 \end{array}%
\end{equation}
and%
\begin{equation}\label{eq_sys2}
  \begin{array}{rl} y^{(1)}_{k+1} & = g_1\left(y_k^{(1)},v_k^{(1)}\right),\qquad (v_k^{(1)}) \in \VC_1\\
  y^{(2)}_{k+1} & = g_2\left(y_k^{(2)},v_k^{(2)}\right),\qquad (v_k^{(2)}) \in \VC_2. \end{array}%
\end{equation}
The corresponding transition maps are denoted by $\varphi_i(k,x_i,\omega_i)$ and $\psi_i(k,y_i,\mu_i)$ ($i=1,2$), resp., the state spaces by $X = X_1 \tm X_2$ and $Y = Y_1 \tm Y_2$, and the control value sets by $U = U_1 \tm U_2$ and $V = V_1 \tm V_2$. We assume that there exists a state transformation $\Phi_1:X_1 \tm U_1 \rightarrow Y_1 \tm V_1$, $\Phi_1(x_1,u_1) = (\alpha(x_1),\beta(u_1))$, and a feedback transformation $\Phi_2:X_2 \tm U_2 \rightarrow Y_2 \tm V_2$, $\Phi_2(x_2,u_2) = (\gamma(x_2),\delta(x_2,u_2))$. Then, if $Q \subset X$ is a controlled invariant set for system \eqref{eq_sys1}, the set $P := (\alpha \tm \gamma)(Q)$ is controlled invariant for system \eqref{eq_sys2} and%
\begin{equation}\label{eq_transformedie}
  h_{\inv}^{(1)}(Q) = h_{\inv}^{(1)}(P).%
\end{equation}
\end{proposition}

\begin{proof}
First note that $P$ is a compact set with nonempty interior, since $\alpha \tm \gamma$ is a homeomorphism. Let $y = (y_1,y_2) \in P$ and put $x = (x_1,x_2) := (\alpha \tm \gamma)^{-1}(y) \in Q$. Since $Q$ is controlled invariant, there exists $u = (u_1,u_2) \in U_1 \tm U_2$ with $(f_1(x_1,u_1),f_2(x_2,u_2)) \in \inner Q$. Put $v_1 := \beta(u_1)$ and $v_2 := \delta(x_2,u_2)$. Then%
\begin{eqnarray*}
  (g_1(y_1,v_1)\!\!\!\!\!\!&,&\!\!\!\!\!\! g_2(y_2,v_2))\\
	&=& (g_1(\alpha(x_1),\beta(u_1)),g_2(\gamma(x_2),\delta(x_2,u_2)))\\
  &=& (\alpha(f_1(x_1,u_1)),\gamma(f_2(x_2,u_2)))\\
  &\in& (\alpha \tm \gamma)(\inner Q) = \inner P.%
\end{eqnarray*}
This proves controlled invariance of $P$.%

Now let $\SC_1\subset\UC_1$ be a $(\tau,Q)^{(1)}$-spanning set and put $\widetilde{\SC}_1 := \{ \beta \circ \omega_1 \}_{\omega_1 \in \SC_1}$. We claim that $\widetilde{\SC}_1$ is a $(\tau,P)^{(1)}$-spanning set. Indeed, take $y = (y_1,y_2) \in P$ and let $(x_1,x_2) := (\alpha \tm \gamma)^{-1}(y_1,y_2)$. Then there exists $(\omega_1,\omega_2) \in \SC_1 \tm \UC_2$ with $(\varphi_1(k,x_1,\omega_1),\varphi_2(k,x_2,\omega_2)) \in \inner Q$ for $k = 1,\ldots,\tau$. Let $\mu_1 := \beta \circ \omega_1 \in \widetilde{\SC}_1$ and $\mu_2(k) :\equiv \delta(\varphi_2(k,x_2,\omega_2),\omega_2(k))$, $\mu_2 \in \VC_2$. Then $(\psi_1(k,y_1,\mu_1),\psi_2(k,y_2,\mu_2)) = (\alpha(\varphi_1(k,x_1,\omega_1)),\gamma(\varphi_2(k,x_2,\omega_2))) \in (\alpha \tm \gamma)(\inner Q) = \inner P$ for $k=1,\ldots,\tau$, proving the claim. Since $\#\widetilde{\SC}_1 = \#\SC_1$, this implies $h_{\inv}^{(1)}(P) \leq h_{\inv}^{(1)}(Q)$. Using that the transformations are invertible, we can interchange the roles of the two networks and obtain the assertion.%
\end{proof}

\begin{remark}
It is not hard to formulate a non-invertible version of the above proposition, in which the transformations are only assumed to be onto and open. In this case, the equality \eqref{eq_transformedie} becomes the inequality $h_{\inv}^{(1)}(P) \leq h_{\inv}^{(1)}(Q)$.%
\end{remark}

\section{Subsystem invariance entropy for linear systems}\label{sec:lin1}%

We can use Proposition \ref{prop_equivalence} to compute the subsystem invariance entropy for linear systems under some controllability assumption and a slightly stronger form of controlled invariance.%

\begin{theorem}\label{thm_lin}
Assume that each of the subsystems $\Sigma_i$ is linear, $x^{(i)}_{k+1} = A_i x_k^{(i)} + B_i u_k^{(i)}$ ($X_i = \R^{d_i}$, $U_i \subset \R^{m_i}$). Fix $i\in\{1,\ldots,n\}$ and assume that for each $j\neq i$ the pair $(A_j,B_j)$ is controllable. Furthermore, assume that there exists a compact set $K \subset \inner Q$ such that every $x\in Q$ can be steered into $\inner K$ in one step of time. Then%
\begin{equation}\label{eq_linearform}
  h_{\inv}^{(i)}(Q) = \sum_{\lambda\in\sigma(A_i)}\max\left\{0,n_{\lambda}\log|\lambda|\right\},%
\end{equation}
where $n_{\lambda}$ denotes the multiplicity of the eigenvalue $\lambda$. In particular, if all subsystems are controllable, then%
\begin{equation}\label{eq_sumform}
  \sum_{i=1}^n h_{\inv}^{(i)}(Q) = h_{\inv}(Q).%
\end{equation}
\end{theorem}

\begin{proof}
By Proposition \ref{prop_elprops}(c), we have $h_{\inv}^{(i)}(Q) \geq h_{\inv}(\pi_i(Q);\Sigma_i)$. Note that $\pi_i(Q)$ has nonempty interior and hence positive Lebesgue measure. Then it follows from a volume growth argument that%
\begin{equation*}
  h_{\inv}(\pi_i(Q);\Sigma_i) \geq \sum_{\lambda\in\sigma(A_i)}\max\left\{0,n_{\lambda}\log|\lambda|\right\},%
\end{equation*}
implying the lower estimate in \eqref{eq_linearform}. The idea of the argument is as follows. The projection of $\Sigma_i$ to the unstable subspace of $A_i$ is a linear system $\Sigma_i^u$ whose trajectories are the projections of those of $\Sigma_i$. If $\pi_i(Q)^u$ is the projection of $\pi_i(Q)$, the invariance entropy of $\pi_i(Q)^u$ is not greater than that of $\pi_i(Q)$. Any $(\tau,\pi_i(Q)^u)$-spanning set naturally is in one-to-one correspondence with a cover of $\pi_i(Q)^u$ whose elements are transformed by the transition map in such a way that their images at time $\tau$ are still contained in $\pi_i(Q)^u$. Then the volume expansion of the transition map of $\Sigma_i^u$ in the $x$-component, which is determined by the unstable determinant of $A_i$, provides a lower bound on the number of elements in this cover, leading to the desired estimate (cf.~\cite[Thm.~5.1]{CKa} or \cite[Thm.~3.1]{Kaw} for more details).%

For the upper estimate, we use the Brunovsky normal form (cf.~\cite[Sec.~5.2]{Son}) for controllable linear systems, together with Proposition \ref{prop_equivalence}. Indeed, we may assume that each of the subsystems $\Sigma_j$, $j\neq i$, is given in Brunovsky normal form and thus has zero eigenvalues. (Here the feedback transformation is linear and has the form $(\gamma(x),\delta(x,u)) = (Tx,Vu-VFx)$ with $T,V$ invertible.) It is easy to see that the strong controlled invariance assumption imposed on $Q$ is preserved by the transformations described in Proposition \ref{prop_equivalence}. Using that $h_{\inv}^{(i)}(Q) \leq h_{\inv}(Q;\Sigma)$ (Proposition \ref{prop_elprops}(c)), it thus suffices to show that%
\begin{equation}\label{eq_toshlin}
  h_{\inv}(Q;\Sigma) \leq \sum_{\lambda \in \sigma(A_1 \oplus A_1 \oplus \cdots \oplus A_n)}\max\left\{0,n_{\lambda}\log|\lambda|\right\}.%
\end{equation}
Using compactness of $Q$ and openness of $\inner K$, one sees that finitely many, say $k$, control values are sufficient to steer from every $x\in Q$ into $\inner K$. Moreover, the set $K$ is controlled invariant and has positive distance $\ep>0$ to the boundary of $Q$. Letting $r_{\inv}(\ep,\tau,K)$ denote the minimal cardinality of a set $\SC\subset\UC$ such that for every $x\in K$ there is $\omega\in\SC$ with $\dist(\varphi(k,x,\omega),K) < \ep$ for $k=1,\ldots,\tau$, we obtain%
\begin{equation*}
  h_{\inv}(Q) \leq \limsup_{\tau\rightarrow\infty}\frac{1}{\tau}\log\left( k \cdot r_{\inv}(\ep,\tau,K) \right).%
\end{equation*}
Obviously, the constant $k$ can be omitted. Therefore, by \cite[Thm.~3.1]{Kaw}, the right-hand side is bounded from above by the right-hand side of \eqref{eq_toshlin}, concluding the proof of \eqref{eq_linearform}. Since in the case $n=1$ the subsystem invariance entropy coincides with the usual invariance entropy, formula \eqref{eq_sumform} immediately follows from \eqref{eq_linearform}.%
\end{proof}

\begin{remark}\
\begin{itemize}
\item The preceding proposition shows that in the given setting the $i$-th subsystem invariance entropy is independent of the specific geometry of the set $Q$ and also of the eigenvalues of the other subsystems $j\neq i$. For nonlinear systems, we expect the situation to be more complicated in general.%
\item Note that the preceding result in the case $n=1$ yields a formula for the invariance entropy of a linear system which in this particular form has not been formulated before. An analogous formula has only been proved for another version of invariance entropy which allows trajectories to leave the set $Q$ and remain in an $\ep$-neighborhood (then the limit for $\ep\searrow0$ is taken).%
\end{itemize}
\end{remark}

\begin{figure}[!h]
\centering
\includegraphics[width=5cm]{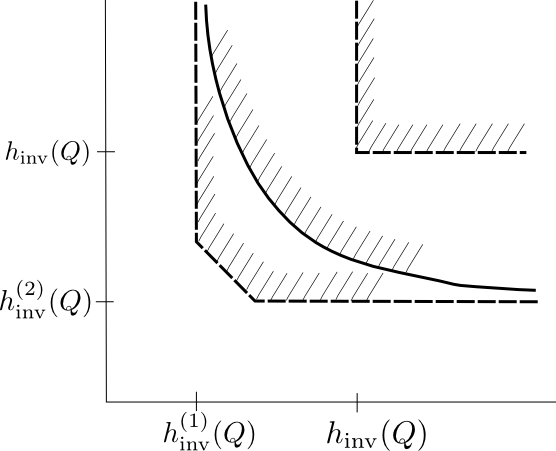}\\
\includegraphics[width=5cm]{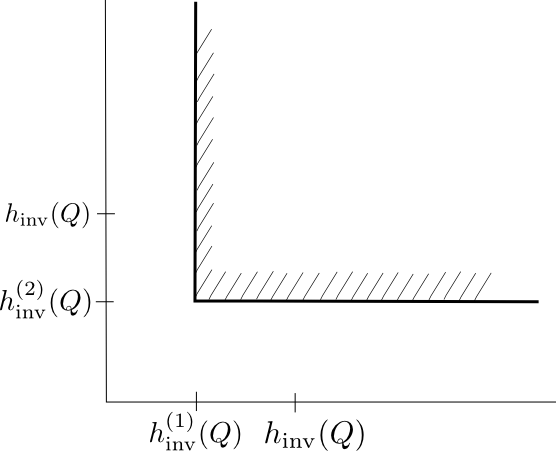}\\
\includegraphics[width=5cm]{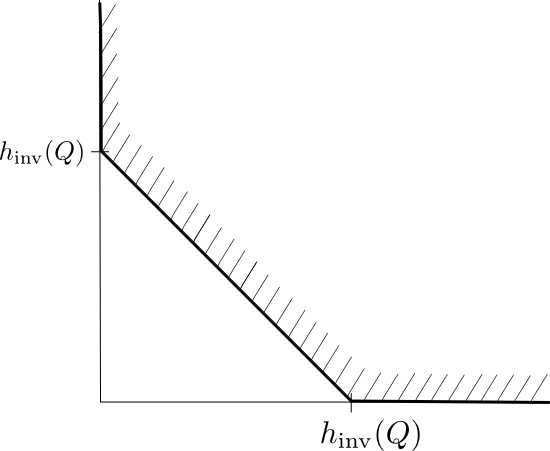}\\
\caption{Top: Lower and upper bounds on network entropy set for two systems. Middle: Network entropy set of a linear system satisfying the conditions of Proposition \ref{prop:HQlin}. Bottom: Network entropy set of synchronization of two chaotic systems.}
\label{fig2}
\end{figure} 

\section{The network entropy set}\label{sec:3}%

In this section, we introduce an object encompassing all possible combinations of data rates for controllers within the given networked system, which allow to make the set $Q$ invariant.%

Consider again the networked system $\Sigma$ of Section \ref{sec:2} with subsystems $\Sigma_i$, $i=1,\ldots,n$. For every time $\tau>0$, define the set%
\begin{eqnarray*}
  H_{\tau}(Q) &:=& \Bigl\{ \frac{1}{\tau}\left(\log\#\SC_1,\ldots,\log\#\SC_n\right): \\
	               && \quad \SC_1\tm\cdots\tm\SC_n \mbox{ finite } (\tau,Q)\mbox{-spanning} \Bigr\},%
\end{eqnarray*}
the elements of which we call \emph{finite-time entropy vectors}.%

\begin{lemma}\label{lem_htau}
The following assertions hold:%
\begin{enumerate}
\item[(a)] $H_{\tau}(Q) \subset H_{s\tau}(Q)$ for all $\tau>0$, $s\in\Z_+$.%
\item[(b)] If $\xi,\eta\in H_{\tau}(Q)$, then $(1/2)(\xi+\eta)\in H_{2\tau}(Q)$.%
\end{enumerate}
\end{lemma}

\begin{proof}
To show (a), let $\xi\in H_{\tau}(Q)$ with corresponding $(\tau,Q)$-spanning set $\SC_1\tm\cdots\tm\SC_n$. For every $i\in\{1,\ldots,n\}$, we consider all possible concatenations of $s$ elements of $\SC_i$, and we denote the set of these control sequences by $\widetilde{\SC}_i$. Then $\#\widetilde{\SC}_i = (\#\SC_i)^s$ and $\widetilde{\SC}_1\tm\cdots\tm\widetilde{\SC}_n$ is an $(s\tau,Q)$-spanning set. This implies%
\begin{equation*}
  \frac{1}{s\tau}\left(\log\#\widetilde{\SC}_1,\ldots,\log\#\widetilde{\SC}_n\right) = \xi \in H_{s\tau}(Q).%
\end{equation*}
To show (b), consider $(\tau,Q)$-spanning sets $\SC_1^{(1)}\tm\cdots\tm\SC_n^{(1)}$ and $\SC_1^{(2)}\tm\cdots\tm\SC_n^{(2)}$ whose associated finite-time entropy vectors are $\xi$ and $\eta$. Let $\SC_i$ be the set of all concatenations of elements of $\SC_i^{(1)}$ and $\SC_i^{(2)}$. Then $\SC_1\tm\cdots\tm\SC_n$ is $(2\tau,Q)$-spanning and $\#\SC_i = \#\SC_i^{(1)}\cdot\#\SC_i^{(2)}$. This implies%
\begin{eqnarray*}
  &&\frac{1}{2\tau}\left(\log\#\SC_1,\ldots,\log\#\SC_n\right)\\
	&&\quad = \frac{1}{2}\Bigl[\frac{1}{\tau}\left(\log\#\SC_1^{(1)},\ldots,\log\#\SC_n^{(1)}\right)\\
     &&\qquad + \frac{1}{\tau}\left(\log\#\SC_1^{(2)},\ldots,\log\#\SC_n^{(2)}\right)\Bigr] = \frac{1}{2}(\xi + \eta),%
\end{eqnarray*}
concluding the proof.%
\end{proof}

We further introduce the set of all limit points of sequences $\xi_k \in H_{\tau_k}(Q)$, where $\tau_k\rightarrow\infty$.%

\begin{definition}
The {\bf network entropy set} of $\Sigma$ is defined as%
\begin{equation*}
  H(Q) := \bigcap_{\tau>0}\cl\bigcup_{t\geq\tau}H_t(Q).%
\end{equation*}
\end{definition}

Obviously, $H(Q)$ is contained in the closed positive orthant of $\R^n$.%

\begin{proposition}\label{prop_entsetprops}
The following assertions hold:%
\begin{enumerate}
\item[(a)] The network entropy set satisfies%
\begin{equation*}
  H(Q) = \cl\bigcup_{\tau>0}H_{\tau}(Q).%
\end{equation*}
In particular, $H(Q)$ is nonempty and closed.%
\item[(b)] Assume that each of the control value sets $U_i$ contains at least two elements. If $\xi\in H(Q)$ and $\eta\geq\xi$ componentwise, then $\eta\in H(Q)$. In particular, $H(Q)$ is unbounded.%
\item[(c)] The set $H(Q)$ is convex.%
\item[(d)] For any $(h_1,\ldots,h_n) \in H(Q)$, it holds that%
\begin{equation*}
   h_{\inv}(Q) \leq \sum_{i=1}^n h_i.%
\end{equation*}
\item[(e)] The set $H(Q)$ contains $(h_{\inv}(Q),\ldots,h_{\inv}(Q))$, provided that $\# U_i \geq 2$ for all $i$.%
\end{enumerate}
\end{proposition}

\begin{proof}
To show (a), note that Lemma \ref{lem_htau}(a) implies $H_{\tau}(Q)\subset H(Q)$ for all $\tau>0$ and hence $\cl\bigcup_{\tau>0}H_{\tau}(Q)\subset H(Q)$, since $H(Q)$ is closed as the intersection of closed sets. On the other hand, by the definition of $H(Q)$ it clearly holds that $H(Q)\subset\cl\bigcup_{\tau>0}H_{\tau}(Q)$.%

To show (b), let $\xi_k\rightarrow\xi$, where $\xi_k\in H_{\tau_k}(Q)$, $\tau_k\rightarrow\infty$. Let $\SC_1^{(k)}\tm\cdots\tm\SC_n^{(k)}$ be a $(\tau_k,Q)$-spanning set with corresponding finite-time entropy vector $\xi_k$. By adding additional control sequences from $\UC_i$ to $\SC_i^{(k)}$ (which is possible by our assumption that $\#U_i\geq2$ and hence $\#\UC_i=\infty$), we can construct $(\tau_k,Q)$-spanning sets $\widetilde{\SC}_1^{(k)}\tm\cdots\tm\widetilde{\SC}_n^{(k)}$ with $(1/\tau_k)(\log\#\widetilde{\SC}_1^{(k)},\ldots,\log\#\widetilde{\SC}_n^{(k)}) \rightarrow \eta$. For instance, this holds if $\#\widetilde{\SC}_i^{(k)} = \max\{\#\SC_i^{(k)},\lfloor 2^{\tau_k\eta_i} \rfloor\}$.%

Now let us show (c). Take $\xi,\eta\in H(Q)$ and let $\xi_k\rightarrow\xi$, $\eta_k\rightarrow\eta$ with $\xi_k\in H_{\tau_k}(Q)$, $\eta_k\in H_{\rho_k}(Q)$, where $\tau_k,\rho_k\rightarrow\infty$. From Lemma \ref{lem_htau}(a) it follows that $\xi_k,\eta_k \in H_{\rho_k\tau_k}(Q)$ for all $k\geq1$. Lemma \ref{lem_htau}(b) implies $(1/2)(\xi_k + \eta_k) \in H_{2\rho_k\tau_k}(Q)$ and thus%
\begin{equation*}
  \frac{1}{2}(\xi + \eta) = \lim_{k\rightarrow\infty}\frac{1}{2}(\xi_k + \eta_k) \in H(Q).%
\end{equation*}
By an iterative argument and closedness of $Q$, it follows that the whole line segment $[\xi,\eta]$ is in $H(Q)$, showing convexity of $H(Q)$.%

To show (d), take $h \in H_{\tau}(Q)$ for some $\tau>0$. Then there exists a finite $(\tau,Q)$-spanning set of the form $\SC_1 \tm \cdots \tm \SC_n$ such that $h_i = (1/\tau)\log\#\SC_i$. It follows that%
\begin{eqnarray*}
   h_{\inv}(Q) &=& \inf_{t>0}\frac{1}{t}\log r_{\inv}(t,Q) \leq \frac{1}{\tau}\log\#(\SC_1\tm\cdots\tm\SC_n)\\
	             &=& \frac{1}{\tau}\log\prod_{i=1}^n\#\SC_i = \sum_{i=1}^n h_i.%
\end{eqnarray*}
Since $H(Q) = \cl\bigcup_{\tau\in\N}H_{\tau}(Q)$, the assertion follows.%

Finally, to show (e), consider a finite $(\tau,Q)$-spanning set $\SC \subset \UC_1\tm\cdots\tm\UC_n$. Since $\SC \subset \SC_1 \tm \cdots \tm \SC_n$ with $\SC_i = \pi_{\UC_i}\SC$, the set $\SC_1 \tm \cdots \tm \SC_n$ is also $(\tau,Q)$-spanning and%
\begin{eqnarray*}
   H(Q) &\ni& \frac{1}{\tau}(\log\#\SC_1,\ldots,\log\#\SC_n)\\
	      &\leq& \frac{1}{\tau}(\log\#\SC,\ldots,\log\#\SC),%
\end{eqnarray*}
which by (b) implies $(1/\tau)(\log\#\SC,\ldots,\log\#\SC) \in H(Q)$ and consequently $(h_{\inv}(Q),\ldots,h_{\inv}(Q))\in H(Q)$ (provided that $\# U_i \geq 2$).%
\end{proof}

The interpretation of the network entropy set is to be found in a data-rate theorem similar to Theorem \ref{thm:datarate1}.%

\begin{theorem}\label{thm:datarate2}
Let $Q$ be a controlled invariant set of $\Sigma$ and fix $i\in\{1,\ldots,n\}$. Then a point $(h_1,\ldots,h_n)$ is in the interior of the network entropy set, $\inner H(Q)$, if and only if there are a control strategy $\Gamma,\Delta$ and channels with zero-error capacities $(h_1,\ldots,h_n)$ that make $Q$ invariant.%
\end{theorem}

The proof, being entirely similar to Theorem \ref{thm:datarate1} (repeating the arguments to all channels simultaneously), is omitted.%

The next proposition relates the network entropy set to the subsystem entropies $h_{\inv}^{(i)}(Q)$.%

\begin{proposition}\label{prop_projinf}
The following statements hold:%
\begin{enumerate}
\item[(a)] For every $i\in\{1,\ldots,n\}$,%
\begin{equation*}
  h_{\inv}^{(i)}(Q) = \inf P_i(H(Q)),%
\end{equation*}
where $P_i:\R^n\rightarrow\R$ is the projection to the $i$-th component.%
\item[(b)] If $Q = Q_1 \tm \cdots \tm Q_n$, then%
\begin{equation*}
   H(Q) = \prod_{i=1}^n [h_{\inv}(Q_i),\infty).%
\end{equation*}
\end{enumerate}
\end{proposition}

\begin{proof}
For the proof of (a), fix $i$ and let $s := \inf P_i(H(Q))$. Then there exists a sequence $\xi_k\in H(Q)$ with $P_i(\xi_k) \rightarrow s$. We can approximate the vectors $\xi_k$ by elements of $\bigcup_{\tau>0}H_{\tau}(Q)$. Hence, we find sequences $\tau_k\rightarrow\infty$ and $\eta_k\in H_{\tau_k}(Q)$ with $P_i(\eta_k) \rightarrow s$. For each $\eta_k$ we have a corresponding $(\tau_k,Q)$-spanning set $\SC_1^{(k)}\tm\cdots\tm\SC_n^{(k)}$. Then $\SC_i^{(k)}$ is $(\tau,Q)^{(i)}$-spanning and $(1/\tau_k)\log\#\SC_i^{(k)}\rightarrow s$, implying%
\begin{eqnarray*}
  h_{\inv}^{(i)}(Q) &=& \lim_{\tau\rightarrow\infty}\frac{1}{\tau}\log r_{\inv}^{(i)}(\tau,Q)\\
	&\leq& \lim_{k\rightarrow\infty}\frac{1}{\tau_k}\log \#\SC_i^{(k)} = \inf P_i(H(Q)).%
\end{eqnarray*}
To show the other inequality, choose for given $\ep>0$ a $\tau>0$ with $(1/\tau)\log r_{\inv}^{(i)}(\tau,Q)-h_{\inv}^{(i)}(Q) \leq \ep$. Let $\SC_i\subset\UC_i$ be a $(\tau,Q)^{(i)}$-spanning set of minimal cardinality $r_{\inv}^{(i)}(\tau,Q)$. We claim that there exist finite sets $\SC_j\subset\UC_j$ ($j\neq i$) such that $\SC_1\tm \cdots \tm \SC_n$ is $(\tau,Q)$-spanning. Indeed, for every $x\in Q$ there exists $\omega = \omega_x \in \UC_1 \tm \cdots \tm \UC_{i-1} \tm \SC_i \tm \UC_{i+1} \tm \cdots \tm \UC_n$ with $\varphi(k,x,\omega_x) \in \inner Q$ for $k=1,\ldots,\tau$. By continuity of $\varphi$ with respect to $x$, there exists an open neighborhood $U_x \subset X$ of $x$ with $\varphi(k,y,\omega_x)\in\inner Q$ for $k=1,\ldots,\tau$ and all $y\in U_x$. By compactness, $Q$ can be covered by finitely many of such neighborhoods, say $U_{x_1},\ldots,U_{x_m}$. The corresponding control sequences $\omega_{x_1},\ldots,\omega_{x_m} \in \UC_1\tm\cdots\tm\UC_{i-1}\tm\SC_i\tm\UC_{i+1}\tm\cdots\tm\UC_n$ form a finite $(\tau,Q)$-spanning set $\SC$, implying the claim (let $\SC_j := \pi_{\UC_j}(\SC)$). Because $H_{\tau}(Q)\subset H(Q)$, this implies%
\begin{eqnarray*}
  h_{\inv}^{(i)}(Q) + \ep &\geq& \frac{1}{\tau}\log r_{\inv}^{(i)}(\tau,Q)\\ &=& P_i\left(\frac{1}{\tau}\left(\log\#\SC_1,\ldots,\log\#\SC_n\right)\right)\\
   &\geq& \inf P_i(H(Q)).%
\end{eqnarray*}
Since this holds for every $\ep$, the proof is complete.%

To prove (b), note that a product set $\SC_1 \tm \cdots \tm \SC_n \subset \UC$ is a finite $(\tau,Q)$-spanning set if and only if each $\SC_i$ is a finite $(\tau,Q_i)$-spanning set. Hence, there exists an element of $H_{\tau}(Q)$ that is minimal componentwise, implying that $H_{\tau}(Q)$ and thus $H(Q)$ is a Cartesian product. Together with statement (a) the assertion follows.%
\end{proof}

Connecting these results, we obtain the following estimate:%
\begin{eqnarray*}
 && \prod_{i=1}^n [h_{\inv}(Q),\infty) \subset H(Q) \\
   && \subset \prod_{i=1}^n [h^{(i)}_{\inv}(Q),\infty) \cap \left\{(h_1, \ldots, h_n): \sum_{i=1}^n  h_i \geq h_{\inv}(Q)\right\}.%
\end{eqnarray*}
See Fig.~\ref{fig2} for a graphical representation.%

\section{The network entropy set for linear systems}\label{sec:4}%

For linear systems, the network entropy set is easy to characterize, under some reasonable assumptions.%

\begin{proposition}\label{prop:HQlin}
For a network of controllable linear systems satisfying the strong invariance condition of Theorem \ref{thm_lin}, and a compact set $Q$ of nonempty interior, the network entropy set is%
\begin{equation*}
  H(Q) = \prod_{i=1}^n \left[\sum_{\lambda\in\sigma(A_i)}\max\left\{0,n_{\lambda}\log|\lambda|\right\},\infty\right).%
\end{equation*}
\end{proposition}

\begin{proof}
We may assume that $0 \in \inner Q$. Let $C = C_1 \tm \cdots \tm C_n$ be a Cartesian product of compact and controlled invariant sets $C_i \subset \R^{d_i}$ with nonempty interiors, such that $0 \in \inner C \subset C \subset \inner Q$. By controllability, every state $x\in Q$ can be controlled to the origin in a finite time $\tau_x$ using a control sequence $u_k(x)$. By compactness of $Q$, we may assume that $\tau_x \leq \tau^*$ for all $x\in Q$ and a constant $\tau^*$. Then we may assume $\tau_x = \tau^*$ for all $x$, because the control sequences $u_k(x)$ can be extended by zeros. By continuity of the transition map in the state variable, a whole neighborhood $N_x$ of $x$ can be steered into $\inner C$ with the same control sequence $u_k(x)$ without leaving $Q$. Hence, there exists a finite set of the form $\RC = \RC_1 \tm \cdots \tm \RC_n \subset \UC$ such that for every $x\in Q$ there is $(v_1,\ldots,v_n) \in \RC$ with $\varphi(k,x,v) \in Q$ for $k=0,1,\ldots,\tau^*$ and $\varphi(\tau^*,x,v) \in C$. Take an element $\xi\in H(C)$. Then there exists a sequence $\xi_k = (1/\tau_k)(\log\#\SC_1^k,\ldots,\log\#\SC_n^k)$, where $\tau_k\rightarrow\infty$ and $\SC_1^k \tm \cdots \tm \SC_n^k$ is $(\tau_k,C)$-spanning such that $\xi_k \rightarrow \xi$. The set $(\SC_1^k \star \RC_1) \tm \cdots \tm (\SC_n^k \star \RC_n)$ is $(\tau_k+\tau^*,Q)$-spanning and%
\begin{equation*}
  \frac{1}{\tau_k + \tau^*}\left(\log\#\SC_1^k + \log\#\RC_1,\ldots,\log\#\SC_n^k + \log\#\RC_n\right)%
\end{equation*}
converges to $\xi$, implying $\xi \in H(Q)$. Hence, $H(C) \subset H(Q)$, and therefore Proposition \ref{prop_projinf}(b) yields%
\begin{equation*}
  \prod_{i=1}^n [h_{\inv}(C_i),\infty) \subset H(Q) \subset \prod_{i=1}^n [h_{\inv}^{(i)}(Q),\infty).%
\end{equation*}
Now $h_{\inv}(C_i) = h_{\inv}^{(i)}(Q) = \sum_{\lambda\in\sigma(A_i)}\max\{0,n_{\lambda}\log|\lambda|\}$ by Theorem \ref{thm_lin}, concluding the proof.%
\end{proof}

\section{The network entropy set for synchronization of chaos}\label{sec:5}%

We now present an example of a control problem where the network entropy set is not rectangular, i.e., a Cartesian product of intervals, but exhibits a trade-off between the data rates required by both subsystems.%

Consider the angle-multiplying system%
\begin{equation*}
  \Sigma:\quad x_{k+1} = (\alpha x_k + u_k) \mbox{ mod } 1%
\end{equation*}
on the unit circle $\rmS^1 = \R/\Z$ with an integer $|\alpha| \geq 2$ and $u_k \in U := [-1,1]$. The natural dynamics of this system, i.e., when $u_k\equiv 0$, is a well-known example of a chaotic system.%

We consider two copies of $\Sigma$ with states $x^{(1)}$ and $x^{(2)}$, which we seek to interconnect in order to reach `practical synchronization', i.e., we want to make the set%
\begin{equation*}
  Q := \left\{ (x^{(1)},x^{(2)}) \in \rmS^1 \tm \rmS^1\ :\  d(x^{(1)},x^{(2)}) \leq \delta \right\}%
\end{equation*}
invariant for a small $\delta>0$, where $d(\cdot,\cdot)$ is the canonical distance on $\R/\Z$, given by $d(x+\Z,y+\Z) = \min_{j\in\Z}|x - y + j|$.%

\begin{theorem}
The entropy set for practical synchronization is given by%
\begin{equation*}
  H(Q) = \left\{ (h_1,h_2) \in \R^2\ :\ h_1,h_2\geq 0,\ h_1 + h_2 \geq \log|\alpha| \right\}.%
\end{equation*}
\end{theorem}

\begin{proof}
For clarity, in the following we write $\bar{x}$ for elements of $\rmS^1$ and $x$ for their representatives in $\R$, i.e., $\bar{x} = x + \Z$. Choosing $\delta$ small enough, we find that the interval $I := [-\delta,\delta] \subset \R$ is controlled invariant for the linear system%
\begin{equation*}
  \Sigma^l:\quad x_{k+1} = \alpha x_k + u_k,\quad u_k \in U,\ x_k \in \R.%
\end{equation*}
Indeed, this holds for every $\delta \leq 1/(2|1-\alpha|)$, because for a given $x \in [-\delta,\delta]$ and small $\ep>0$ the control input $u_x := (1 - \alpha)x \pm \ep$ satisfies $|u_x| < 1/2 + \ep \in U$ and $\alpha x + u_x = x \pm \ep$. Moreover, if $\varphi^l(k,x,\omega)$ and $\varphi(k,\bar{x},\omega)$ denote the transition maps of $\Sigma^l$ and $\Sigma$, resp., then%
\begin{equation*}
  \varphi(k,\bar{x},\omega) \equiv \varphi(k,x,\omega) + \Z.%
\end{equation*}
We claim that every $(\tau,I)$-spanning set $\SC$ for $\Sigma^l$ yields a $(\tau,Q)$-spanning set of the same cardinality for the product system on $\rmS^1 \tm \rmS^1$, given by%
\begin{equation*}
  \SC' := \left\{ (\omega,0) \ :\ \omega \in \SC \right\}.%
\end{equation*}
Indeed, if $(\bar{x}^{(1)},\bar{x}^{(2)}) \in Q$, we may assume that the representatives $x^{(1)},x^{(2)}\in\R$ are chosen such that $x^{(1)} - x^{(2)} \in [-\delta,\delta]$. Then there exists $\omega \in \SC$ such that $\varphi^l(k,x^{(1)} - x^{(2)},\omega)\in (-\delta,\delta)$ for $k=1,\ldots,\tau$. Hence,%
\begin{eqnarray*}
  \varphi\left(k,\bar{x}^{(1)},\omega\right) \!\!\!&-&\!\!\! \varphi\left(k,\bar{x}^{(2)},0\right) \mbox{ mod } 1\\
	&=& \varphi^l\left(k,x^{(1)},\omega\right) - \varphi^l\left(k,x^{(2)},0\right) \mbox{ mod } 1\\
	&=& \varphi^l\left(k,x^{(1)}-x^{(2)},\omega\right) \mbox{ mod } 1,%
\end{eqnarray*}
implying, for $k=1,\ldots,\tau$,%
\begin{eqnarray*}
  && d\left(\varphi\left(k,\bar{x}^{(1)},\omega\right),\varphi\left(k,\bar{x}^{(2)},0\right)\right)\\
	&&  = \min_{j\in\Z}\left|\varphi^l\left(k,x^{(1)}-x^{(2)},\omega\right) + j\right| < \delta,%
\end{eqnarray*}
which proves the claim. Consequently, $h_{\inv}(Q) = \log|\alpha|$ (using Theorem \ref{thm_lin} with $n=1$). By what we have just shown, $(\tau,Q)$-spanning sets of the form $\SC \tm \{0\}$ exist, which immediately implies $h_{\inv}^{(2)}(Q) = 0$. Using that $\SC$ grows asymptotically like $2^{h_{\inv}(I;\Sigma^l)} = 2^{\log |\alpha|}$, we obtain $(\log|\alpha|,0) \in H(Q)$. By symmetry, $h^{(1)}_{\inv}(Q) = 0$ and $(0,\log|\alpha|)\in H(Q)$. Together with Proposition \ref{prop_entsetprops}(b,c,d), this proves the theorem.%
\end{proof}

The theorem is illustrated on bottom of Fig.~\ref{fig2}. Therefore, we observe a trade-off in the data rates required by each of the two subsystems: one may receive less, or even no information from the state of the overall system, provided that the other subsystem receives more.%

\section{Perspectives}\label{sec:6}%

We see the current framework as both interesting in itself, at the theoretical and applicative level, and a stepping stone to a more ambitious framework including prescribed rates between any pair of systems, both in a deterministic and in a stochastic framework. Such extensions are by no means trivial, as for instance a general stochastic version for the invariance entropy (or any similarly general tool) for even a single system is not known to this date. A better understanding of the (necessarily nonlinear) situations where a trade-off between the different data rates is allowed to the different subsystems, as we showed for synchronization of chaos, is also desirable.%

\bibliography{subsent_bib}

\end{document}